\documentclass{amsart}
\usepackage{amsmath, amssymb, amsthm, amsfonts}
\usepackage{tikz}

 \newtheorem{theorem}{Theorem}
 
  \newtheorem{prop}[theorem]{Proposition}
 \newtheorem{coro}[theorem]{Corollary}

\theoremstyle{definition}

\newtheorem*{remark}{Remark}
\newtheorem{eje}{Example}

\begin{document}

\title{Where did the examples of Abel's continuity theorem go?}

\author{Sergio A. Carrillo}
\address{Programa de Matem\'{a}ticas, Universidad Sergio Arboleda, Calle 74, $\#$ 14-14, Bogot\'{a}, Colombia.}\email{sergio.carrillo@usa.edu.co}

\keywords{Abel's lemma, sum of series, special functions}

\thanks{Supported by Ministerio de Econom\'{i}a y Competitividad from Spain, under the Project ``M\'{e}todos asint\'{o}ticos, algebraicos y geom\'{e}tricos en foliaciones singulares y sistemas din\'{a}micos" (Ref.: PID2019-105621GB-I00) and Univ. Sergio Arboleda project ``An\'{a}lisis complejo, ecuaciones diferenciales y sumabilidad" (IN.BG.086.20.002).}

\subjclass[2010]{Primary: 30B30, 40A05. Secondary: 33E20}

\maketitle

\begin{abstract}
Abel's continuity theorem for power series is a great tool to compute sums and prove properties of special functions. However, apart from two basic examples, this no longer occupies a place in Analysis courses. This note is an invitation to get acquainted with this theorem, its history, and many of its applications dispersed in the literature.
\end{abstract}

\section{Introduction.}

It is very likely while studying a first course on Real Analysis \cite[Thm. 8.2]{Rudin real} or Complex
Function Theory \cite[p. 41]{Ahlfors}, you learned the main character of this note, namely,

\begin{theorem}[Abel]\label{T Abel} If the series of complex numbers $\sum_{n=0}^\infty a_n$ converges to $s$, then $f(x)=\sum_{n=0}^\infty a_n x^n$ converges for $-1<x<1$ and $\lim_{x\rightarrow 1^-} f(x)=s.$
\end{theorem}

Usually, it is taught in an extended form, due to O. Stolz, asserting we can take $x\in\mathbb{C}$ with $|x|<1$, and the limit holds if $|1-x|/(1-|x|)$ remains bounded. Quickly after
proving it, you find the converse is false: Grandi's series  $1-1+1-1+\cdots$ diverges, but $1-x+x^2-x^3+\cdots=\frac{1}{1+x}$ goes to $1/2$ as $x\to 1^-$. Enthusiastic teachers mention the series is \textit{Abel summable} to $1/2$. Later, it is certain that through the power series $
\log(1+x)=\sum_{n=0}^\infty \frac{(-1)^n}{n+1} x^{n+1}$ and $\arctan(x)=\sum_{n=0}^\infty \frac{(-1)^n}{2n+1} x^{2n+1}$, 
convergent for $|x|<1$, the first examples you meet are \begin{equation}\label{Eq. ln 2 pi4}
1-\frac{1}{2}+\frac{1}{3}-\frac{1}{4}+\cdots=\ln(2),\qquad 1-\frac{1}{3}+\frac{1}{5}-\frac{1}{7}+\cdots=\arctan(1)=\frac{\pi}{4}.\end{equation} 

The latter is the Leibniz formula for $\pi$ (1676), but first discovered by Madhava of the Indian Kerala school in the 14th century. Now, if you were discussing the Cauchy product
of series, the second application is another result of Abel: if $\sum_{n=0}^\infty a_n$, $\sum_{n=0}^\infty b_n$ and $\sum_{n=0}^\infty (\sum_{k=0}^n a_k b_{n-k})$ converge to $a, b$ and $c$, respectively, then $c=ab$ \cite{Rudin real}. 

But, what about some other elementary examples? After all, it is always stressed Abel's theorem is such a powerful tool. The answer will depend on your experience, but unfortunately the response from young students may not be very encouraging.

This note is aimed to contribute on this point by collecting an array of examples at a basic level, showing the reach of Theorem \ref{T Abel}. The five examples in the first two sections go from Euler's first contribution to Basel problem, to values of Dirichlet $L$-functions, and they can be read independently. Then, we address another one under
the lenses of special functions, including a definite integral due to Gauss. This integral is based on Theorem \ref{Main Thm.}, whose proof is sketched at the end as an invitation for the readers to roll up their sleeves and make some calculations!

\section{A bit of history. First examples}

Theorem \ref{T Abel} was proved by N. Abel in his 1826 paper
\cite{Abel} (\textit{Lehrsatz VI}, page 316), see \cite[Ch.16]{Stedall} for a translation into English. His goal was to ``fill a gap" in the proof of \begin{equation}\label{Ec. Binomial}
(1+x)^m=1+\sum_{n=1}^\infty \frac{m(m-1)\cdots (m-n+1)}{n!}x^n,
\end{equation} ``for all real or imaginary values of $x$ and $m$ for which the binomial series is convergent''. The work was based on Cauchy's rigorous \textit{Cours d'analyse de l'\'{e}cole polytechnique}, who had only considered the problem for real $m$. Nowadays, we learn from an early stage that if $m$ is not a positive integer, this series converges for $|x|<1$ (ratio test). Then, we find its sum solving the initial-value problem $(1+x)y'(x)=m\cdot y(x)$, $y(0)=1$ \cite[p. 201]{Rudin real}. What happens on the boundary $|x|=1$? That was precisely the purpose with Theorem \ref{T Abel}! as (\ref{Ec. Binomial}) will hold as long as the convergence is established. The answer is: for $x\neq -1$, the series converges for $\text{Re}(m)>-1$. For $x=-1$ the
same is true for $\text{Re}(m)>0$, see \cite[p.400--402]{Mark} for an accessible proof.

Abel's achievement was not only to have settled down the long-standing problem on giving a rigorous treatment of the binomial theorem \cite{Coolidge}. It was also the starting point for further research on the theory of power series. For instance, Dirichlet provided another proof of Theorem \ref{T Abel} using the test that today bears his name\footnote{Dirichlet test: if $(a_{n})$ is a monotonically decreasing sequence of real numbers such that $a_n\to 0$ as $n\to+\infty$, and $(b_n)$ a sequence of complex numbers with bounded partial sums, then $\sum_{n=0}^\infty a_nb_n$ converges. The case $b_n=(-1)^{n}$ is precisely Leibniz's test.}. On the other hand, it inspired a deeper study of continuity for series of functions. In fact, Abel detected an \textit{exception} to Cauchy's erroneous theorem that an infinite sum of continuous functions is itself continuous \cite{Sorensen}, that had to wait to be corrected after the notion of uniform convergence was given by Weierstrass \cite[p. 524--527]{Stedall}. The reader may also find it instructive to read Hardy's discussion on what Theorem \ref{T Abel} brought \cite{Hardy1907}.

Let's turn to examples now, starting with three basic ones to warm up: one comes from the Basel problem, the second one from rearrangements of conditional convergent series, and the last one is borrowed from Fourier series.

\begin{eje}\label{Dilog} The Basel problem consisted of finding the exact value of the sum $\zeta(2)=\sum_{n=1}^\infty 1/n^2$. The great achievement $\zeta(2)=\pi^2/6$ was obtained by L. Euler in 1734 who gave different proofs in the years to come \cite{Ayoub,Varadarajan}. However, his first contribution occurred in 1731 with a numerical approximation. He showed that \begin{equation}\label{Eq. zeta2}
	\zeta(2)=\sum_{n=1}^\infty \frac{x^n}{n^2}+\sum_{n=1}^\infty \frac{(1-x)^n}{n^2}+\log(x)\log(1-x),\quad 0<x<1,
	\end{equation} and setting $x=1/2$ he found $\zeta(2)=(\ln 2)^2+\sum_{n=1}^\infty \frac{1}{2^{n-1}n^2}$. But this and the series $\ln(2)=-\ln(1-1/2)=\sum_{n=1}^\infty \frac{1}{n 2^n}$ are rapidly convergent, which led him to the value $\zeta(2)\approx 1.644934$ correct to $6$ decimal places.  Nowadays, (\ref{Eq. zeta2}) is recognized as one of the functional equations of the \textit{dilogarithm} $\text{Li}_2(x):=\sum_{n=1}^\infty \frac{x^n}{n^2}$, $|x|<1$ \cite[Sec. 2.6]{Andrews}, which produces the special value $$\text{Li}_2\left(1/2\right)=\sum_{n=1}^\infty \frac{1}{2^{n}n^2}=\frac{\pi^2}{12}-\frac{(\ln 2)^2}{2}.$$ Although Abel's theorem was not used to solve Basel problem, it plays a part in
	the story as it provides a proof of (\ref{Eq. zeta2}).  Indeed, its right hand side is constant as we may
	check by taking its derivative and finding it is zero. Letting $x\to 1^-$, this constant is $$\lim_{x\to 1^-} \sum_{n=1}^\infty \frac{x^n}{n^2}+\sum_{n=1}^\infty \frac{(1-x)^n}{n^2}+\log(x)\log(1-x)=\sum_{n=1}^\infty \frac{1}{n^2}=\zeta(2),$$ as claimed, since $\log(x)\log(1-x)\to 0$ as $x\to 1^-$. 	
\end{eje}

\begin{eje} Riemann's rearrangement theorem asserts that a conditionally convergent series of real numbers can be rearranged to sum to any real value, including $\pm\infty$ \cite[Thm. 3.54]{Rudin real}. Giving examples is not an easy task and usually they are confined to the alternating harmonic series $\sum_{n=1}^\infty (-1)^n/n$. The common approach is based on the existence of the limit \begin{equation}\label{Eq. EulerMascheroni}
	\gamma:=\lim_{ N\to+\infty} 1+\frac{1}{2}+\cdots+\frac{1}{N}-\ln(N)=0.577215\dots,
	\end{equation} known as the \textit{Euler-Mascheroni} constant. However, Theorem \ref{T Abel} comes in handy for calculating some sums of rearrangements when used correctly. Let us illustrate this point with $$1+\frac{1}{3}-\frac{1}{2}+\frac{1}{5}+\frac{1}{7}-\frac{1}{4}+\frac{1}{9}+\frac{1}{11}-\frac{1}{6}+\cdots.
	$$	If we jump directly to consider its generating series, we end up with the function \begin{equation}\label{Eq. Ex2}
	1+\frac{x}{3}-\frac{x^2}{2}+\frac{x^3}{5}+\frac{x^4}{7}-\frac{x^5}{4}+\cdots=\sum_{n=0}^\infty \frac{x^{3n}}{4n+1}+\frac{x^{3n+1}}{4n+3}-\frac{x^{3n+2}}{2n+2}.
	\end{equation}
	We can find a ``closed" expression for it using the Gauss hypergeometric series \begin{equation}\label{Eq. Hyper}
	{}_{2}F_{1}(a,b;c;z):=1+\sum _{n=1}^{\infty }{\frac {(a)_{n}(b)_{n}}{(c)_{n}}}{\frac {z^{n}}{n!}},
	\end{equation} where $(q)_{n}:=q(q+1)\cdots (q+n-1)$, $n>0$, stands for the Pochhammer symbol. Here $a,b,c$ can be complex numbers only restricted to $c\neq 0,-1,-2,\dots$. The ratio test shows (\ref{Eq. Hyper}) converges for $|z|<1$, assuming $a,b\neq 0,-1,-2,\dots$, otherwise ${}_2F_1$ reduces to a polynomial. Now, noticing that $(q+n)\cdot (q)_n=q\cdot (q+1)_n$ and $(1)_n=n!$, we find ${}_{2}F_{1}(a,1;a+1;z)=a \sum_{n=0}^\infty \frac{z^n}{n+a}$. Therefore, (\ref{Eq. Ex2}) is equal to
	$${}_{2}F_{1}\left({1}/{4},1;{5}/{4};x^3\right)+\frac{x}{3}  {}_{2}F_{1}\left({3}/{4},1;{7}/{4};x^3\right)+\frac{\log(1- x^3)}{2x},$$ but letting $x\to 1^-$ seems unlikely without further knowledge on these functions. However, we can place zeros adequately to find a more familiar  generating series \cite{Alternating}. If we consider $1+0+\frac{1}{3}-\frac{1}{2}+\frac{1}{5}+0+\frac{1}{7}-\frac{1}{4}+\frac{1}{9}+0+\frac{1}{11}-\frac{1}{6}+\cdots,$ where the zeros are placed in the even positions which are not multiples of $4$, we get \begin{align*}
	&x+\frac{x^3}{3}-\frac{x^4}{2}+\frac{x^5}{5}+\frac{x^7}{7}-\frac{x^8}{4}+\cdots=\sum_{n=0}^\infty \frac{x^{2n+1}}{2n+1}-\sum_{n=1}^\infty \frac{x^{4n}}{2n}\\
	&=\frac{1}{2}\left[\log(1+x)-\log(1-x)+\log(1-x^4)\right]
	=\frac{1}{2}\log((1+x)^2(1+x^2)),
	\end{align*} valid for $|x|<1$. Finally, letting $x\to1^-$ leads to the value $\frac{3}{2}\ln(2)$. The same procedure works for rearrangements
	in which blocks of $p$ positive terms alternate with blocks of $q$ negative terms. The reader can verify that the function $$\frac{1}{2}[\log(1+x^q)-\log(1-x^q)+\log(1-x^{2p})]
	=\frac{1}{2}\cdot \log\left[(1+x^q)\frac{1-x^{2p}}{1-x^q}\right]$$ does the trick, and by taking $x\to 1^-$, we get the value $\ln(2)+\frac{1}{2}\ln(p/q)$.
\end{eje}

\begin{remark} Placing zeros in a convergent series doesn't change convergence or the value of the sum. Indeed, simply notice the sequence of partial sums of the series with additional zeros repeats a finite number of times each member of the original sequence of partial sums. Additionally, if a series is absolutely convergent, we can rearrange or group its terms as we please. These facts will be used without further mention.
\end{remark}

\begin{eje}\label{Example cos(nt)/n sin(nt)/t} The theory of Fourier series offers tools to compute sums such as Parseval's identity. Humbly, Abel's theorem contributes here as well. The classical example is  $\sum_{n=0}^\infty e^{i(n+1)\theta}/(n+1)$, for $0<|\theta|<\pi$. The convergence follows from Dirichlet's test: the partial sums $\sum_{j=1}^k e^{ij\theta}$ are bounded because  \begin{equation}\label{Eq. sum eit}
	\sum_{j=1}^k e^{ji\theta}=\frac{e^{ki\theta }-1}{e^{i\theta}-1}e^{i\theta}=\frac{e^{ki\theta/2}-e^{-ki\theta/2}}{e^{i\theta/2}-e^{-i\theta/2}}e^{i\theta/2}e^{ki\theta/2}=\frac{\sin\left(k\theta/2\right)}{\sin\left(\theta/2\right)}e^{(k+1)i\theta/2},
	\end{equation} and $|\sum_{j=1}^k e^{ij\theta}|\leq 1/|\sin(\theta/2)|$.  Then, Theorem \ref{T Abel} shows that \begin{equation}\label{Eq. eitn}
	\sum_{n=0}^\infty \frac{e^{i
			(n+1)\theta}}{n+1}=\lim_{x\rightarrow 1^-} \sum_{n=0}^\infty \frac{x^{n+1}e^{i(n+1)\theta}}{n+1}=-\log(1-e^{i\theta}),
	\end{equation} where now $\log(z)=\ln|z|+i\text{arg}(z)$ is the principal branch, i.e.,  $|\text{arg}(z)|<\pi$. The middle series can also be written as $$e^{i\theta}\cdot  \sum_{n=0}^\infty \int_0^x  e^{in\theta}t^n dt=e^{i\theta}\cdot \int_0^x \frac{dt}{1-e^{i\theta}t}=\int_0^x \frac{dt}{e^{-i\theta}-t},$$ where term by term integration is allowed due to the uniform convergence of the geometric series for $|t|\leq x<1$. Thus, taking $x\to1^-$ and changing $\theta$ by $-\theta$, proves
	\begin{equation}\label{Eq. Lema2}
	\int_0^1 \frac{dt}{t-e^{i\theta}}=\ln(2)+\ln\left(\sin\frac{|\theta|}{2}\right)+i\cdot \text{arg}(1-e^{-i\theta}),\qquad 0<|\theta|<\pi,
	\end{equation} which will be useful later. On the other hand, equating real and imaginary parts in (\ref{Eq. eitn}), 
	$$\sum_{n=1}^\infty \frac{\cos(n\theta)}{n}=-\ln(2)-\ln\left(\sin\frac{|\theta|}{2}\right),\quad  \sum_{n=1}^\infty \frac{\sin(n\theta)}{n}=-\text{arg}(1-e^{i\theta}).$$ The values in (\ref{Eq. ln 2 pi4}) are the cases $\theta=\pi$ and $\theta=\pi/2$, respectively. Here we used  \begin{equation}\label{Eq. mod 1-ei0}
	|1-e^{i\theta}|^2=(1-\cos(\theta))^2+\sin(\theta)^2=2(1-\cos(\theta))=4\sin(\theta/2)^2.
	\end{equation} This can also be justified using the law of sines in Figure \ref{fig:1}, which in turn shows that
	\begin{equation}\label{Eq. Arg}
	\text{arg}(1-e^{i\theta})=\frac{\theta-\pi}{2}, \text{  if } 0<\theta<\pi,\text{ and equal to } \frac{\pi+\theta}{2}, \text{ if} -\pi<\theta<0.
	\end{equation}

	\begin{figure}[h]
		\tikzset{every picture/.style={line width=0.65pt}} 

		\begin{tikzpicture}[x=0.85pt,y=0.85pt,yscale=-.8,xscale=.8]
		\tikzset{every picture/.style={line width=0.75pt}}

		\draw  (59.25,133.74) -- (273.45,133.74)(164.85,24.07) -- (164.85,238.27) (266.45,128.74) -- (273.45,133.74) -- (266.45,138.74) (159.85,31.07) -- (164.85,24.07) -- (169.85,31.07)  ;
		\draw   (85.81,133.74) .. controls (85.81,90.08) and (121.2,54.7) .. (164.85,54.7) .. controls (208.5,54.7) and (243.89,90.08) .. (243.89,133.74) .. controls (243.89,177.39) and (208.5,212.78) .. (164.85,212.78) .. controls (121.2,212.78) and (85.81,177.39) .. (85.81,133.74) -- cycle ;
		\draw    (118.5,70.57) -- (164.85,133.74) ;
		\draw  [draw opacity=0] (181.27,133.47) .. controls (181.13,124.74) and (173.71,117.75) .. (164.66,117.86) .. controls (161.2,117.9) and (158.01,118.97) .. (155.38,120.77) -- (164.85,133.74) -- cycle ; \draw   (181.27,133.47) .. controls (181.13,124.74) and (173.71,117.75) .. (164.66,117.86) .. controls (161.2,117.9) and (158.01,118.97) .. (155.38,120.77) ;
		\draw  (321.77,133.44) -- (547.8,133.44)(372.81,24.9) -- (372.81,236.79) (540.8,128.44) -- (547.8,133.44) -- (540.8,138.44) (367.81,31.9) -- (372.81,24.9) -- (377.81,31.9)  ;
		\draw   (372.81,133.44) .. controls (372.81,89.78) and (408.2,54.4) .. (451.85,54.4) .. controls (495.5,54.4) and (530.89,89.78) .. (530.89,133.44) .. controls (530.89,177.09) and (495.5,212.48) .. (451.85,212.48) .. controls (408.2,212.48) and (372.81,177.09) .. (372.81,133.44) -- cycle ;
		\draw  [draw opacity=0] (445.71,134.28) .. controls (446.29,138.35) and (449.82,141.45) .. (454.05,141.4) .. controls (455.84,141.38) and (457.5,140.79) .. (458.84,139.81) -- (453.95,133.09) -- cycle ; \draw   (445.71,134.28) .. controls (446.29,138.35) and (449.82,141.45) .. (454.05,141.4) .. controls (455.84,141.38) and (457.5,140.79) .. (458.84,139.81) ;
		\draw  [draw opacity=0] (393.21,140.98) .. controls (393.48,140.36) and (393.7,139.71) .. (393.84,139.02) .. controls (394.26,137.07) and (394.05,135.14) .. (393.35,133.47) -- (385.72,137.3) -- cycle ; \draw   (393.21,140.98) .. controls (393.48,140.36) and (393.7,139.71) .. (393.84,139.02) .. controls (394.26,137.07) and (394.05,135.14) .. (393.35,133.47) ;
		\draw    (164.85,133.74) -- (191.3,167.2) -- (214.9,195.25) ;
		\draw  [dash pattern={on 0.84pt off 2.51pt}]  (372.81,133.44) -- (503.9,194.25) ;
		\draw    (453.95,133.09) -- (502.9,193.25) ;
		
		\draw (270,138) node [anchor=north west][inner sep=0.75pt]   [align=left] {$\displaystyle x$};
		\draw (172,17) node [anchor=north west][inner sep=0.75pt]   [align=left] {$\displaystyle y$};
		\draw (172,107) node [anchor=north west][inner sep=0.75pt]  [font=\scriptsize] [align=left] {$\displaystyle \theta $};
		\draw (105,51) node [anchor=north west][inner sep=0.75pt]  [font=\footnotesize] [align=left] {$\displaystyle e^{i\theta }$};
		\draw (245.89,136.74) node [anchor=north west][inner sep=0.75pt]  [font=\scriptsize] [align=left] {$\displaystyle 1$};
		\draw (381,17.7) node [anchor=north west][inner sep=0.75pt]   [align=left] {$\displaystyle y$};
		\draw (442,140.97) node [anchor=north west][inner sep=0.75pt]  [font=\tiny] [align=left] {$\displaystyle \theta $};
		\draw (506,193.7) node [anchor=north west][inner sep=0.75pt]  [font=\footnotesize] [align=left] {$\displaystyle 1-e^{i\theta }$};
		\draw (449.02,121.6) node [anchor=north west][inner sep=0.75pt]  [font=\scriptsize] [align=left] {$\displaystyle 1$};
		\draw (543,137) node [anchor=north west][inner sep=0.75pt]   [align=left] {$\displaystyle x$};
		\draw (400.25,135.97) node [anchor=north west][inner sep=0.75pt]  [font=\tiny] [align=left] {$\frac{\pi -\theta }{2}$};
		\draw (209.9,194.25) node [anchor=north west][inner sep=0.75pt]  [font=\footnotesize] [align=left] {$\displaystyle -e^{i\theta }$};
		
		\end{tikzpicture}
		\caption{The determination of $\text{arg}(1-e^{i\theta})$ using an isosceles triangle.}
		\label{fig:1}	
	\end{figure}
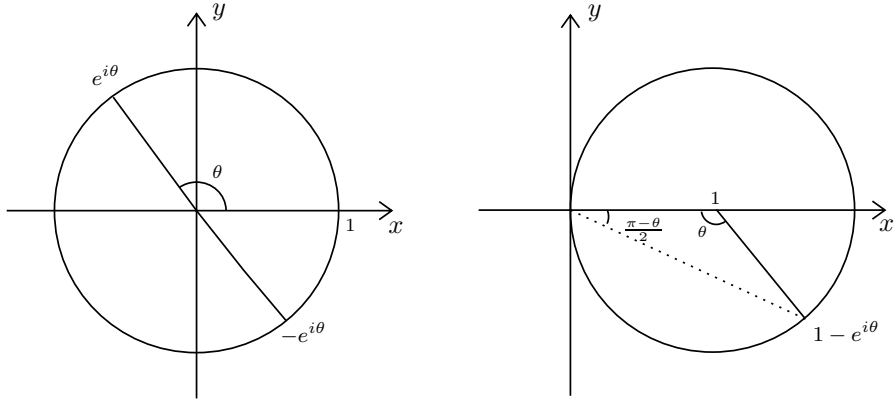

	We leave to the reader as exercise to check by the same method that $$\sum_{n=1}^\infty (-1)^{n+1}\frac{\cos(n\theta)}{n}=\ln\left(2\cos\frac{\theta}{2}\right),\quad  \sum_{n=1}^\infty (-1)^{n+1}\frac{\sin(n\theta)}{n}=\frac{\theta}{2},\quad |\theta|<\pi.$$ Note that the second function breaks if $\theta=\pi$! This was precisely Abel's counterexample to Cauchy's continuity assertion on series of functions. Furthermore, the attentive reader may recognize Abel's example (\ref{Ec. Binomial}) for $x=e^{i\theta}$ as the Fourier series $$\sum_{n=0}^\infty \binom{\alpha}{n} e^{in\theta}=(1+e^{i\theta})^{\alpha}=\exp(\alpha\log(1+e^{i\theta}))=\left[2\cos\left({\theta}/{2}\right)\right]^{\alpha} e^{i\alpha \theta/2},$$ where $\binom{\alpha}{n}=\frac{\alpha(\alpha-1)\cdots(\alpha-n+1)}{n!}$, for $\text{Re}(\alpha)>-1$ and $|\theta|<\pi$, as mentioned in page 2.
\end{eje}

\section{Two examples from definite integrals}

Another source of examples are series as definite integrals, and we can gladly rely on classical books like Bromwich's \cite[p. 189-190]{Bromwich} or Knopp's \cite[p.236, p.269]{Knopp} to bring them here. We have chosen two families of series, which were familiar to Euler and
Gauss, and remarkably used by Dirichlet in Number Theory.

\begin{eje}\label{Ex. Knopp} The first collection comprises alternating series of the form
	\begin{align*}
	1-\frac{1}{4}+\frac{1}{7}-\frac{1}{10}&+\cdots=\frac{\ln 2}{3}+ \frac{\pi}{3\sqrt{3}},\quad \frac{1}{2}-\frac{1}{5}+\frac{1}{8}-\frac{1}{11}+\cdots=\frac{\pi}{3\sqrt{3}}-\frac{\ln 2}{3},\\
	&1-\frac{1}{5}+\frac{1}{9}-\frac{1}{13}+\cdots=\frac{\pi}{4\sqrt{2}}+\frac{\ln(1+\sqrt{2})}{2\sqrt{2}}.
	\end{align*} We can recognize them to have the form $\sum_{n=0}^{\infty} \frac{(-1)^n}{pn+q}$, for integers $p,q>0$. The fact these are definite integrals in disguise is a consequence of: 
	
	\begin{prop}\label{Prop. Sum p} If $\lambda>0$, then $\displaystyle  I(\lambda):=\sum_{n=0}^\infty \frac{(-1)^n}{\lambda n+1}=\int_0^1 \frac{dt}{1+t^\lambda}.$
	\end{prop}
	
	\begin{proof} Convergence follows from Leibniz's test. Now, Theorem \ref{T Abel} shows $I(\lambda)$ is $$\lim_{x\to1^-} \sum_{n=0}^\infty \frac{(-1)^n}{\lambda n+1}x^{n}=\lim_{u\to1^-} \sum_{n=0}^\infty \frac{(-1)^n}{\lambda n+1}u^{\lambda n+1}=\lim_{u\to1^-} \sum_{n=0}^\infty (-1)^n \int_0^u  t^{\lambda n} dt,$$
		where we changed variables $x=u^\lambda$ in the first equality and add an extra $u$. But the geometric series converges
		uniformly in closed intervals of $[0,1)$, thus the interchange of sums and integrals is valid, leading to the result.
	\end{proof}
	
	For rational values $\lambda=p/q>0$, the change of variable $s=t^q$ gives 
	\begin{equation}\label{Eq. 2int I}
	I(p/q)/q=\sum_{n=0}^\infty \frac{(-1)^n}{pn+q}=\frac{1}{q}\int_0^1 \frac{ds}{1+s^{p/q}}=\int_0^1 \frac{t^{q-1}}{1+t^p}dt=\int_0^1 \frac{t^{q-1}-t^{p+q-1}}{1-t^{2p}}dt,
	\end{equation} and finding the sum is reduced to computing (\ref{Eq. 2int I}). Another fact on this
	integral is that $I(0)={1}/{2}$ now makes sense and recovers the Abel sum of Grandi's series.
\end{eje}

\begin{eje}\label{Ex. Diric} The second family of series have the form
	\begin{align*}
	\label{Eq.Ex5}	&1+\frac{1}{2}-\frac{1}{4}-\frac{1}{5}+\frac{1}{7}+\frac{1}{8}-\cdots=\frac{2\pi}{3\sqrt{3}},\\
	\nonumber	1-\frac{1}{2}+\frac{1}{4}-\frac{1}{5}+\frac{1}{7}&-\frac{1}{8}+\cdots=\frac{\pi}{3\sqrt{3}},\, 1+\frac{1}{3}-\frac{1}{5}-\frac{1}{7}+\frac{1}{9}+\frac{1}{11}-\cdots=\frac{\pi}{2\sqrt{2}},\end{align*} and they appeared in Euler's studies on integrals related with trigonometric functions. The latter two are particular cases, for $s=1$, of $L$-functions $L(s,\chi):=\sum_{n=1}^\infty \frac{\chi(n)}{n^s}$, associated to characters $\chi$. These were used by Dirichlet in his famous theorem on the infinitude of primes in arithmetic
	progressions \cite{Apostol,Stein}.  Briefly, a \textit{Dirichlet character} is a function $\chi:\mathbb{Z}\to\mathbb{C}$ which is $p$-periodic for some $p>0$, completely multiplicative ($\,\chi(mn)=\chi(m)\chi(n)$\, for all integers $m,n$), and satisfying  $\chi(n)\neq 0$ if and only if $\text{g.c.d.}(n,p)=1$. In our case, the periods are $p=3$ and $8$, and $\chi$ assumes the
	values $(1,-1,0)$ and $(1,0,1,0,-1,0,-1,0)$, respectively. But,
	the key point here is non-trivial characters $\chi$, i.e., those not being identically $1$, satisfy $$\chi(1)+\chi(2)+\cdots+\chi(p-1)=0.$$ Therefore, the step from series to integrals is a consequence of:
\end{eje}

\begin{prop}\label{Pr. 2} Consider a polynomial $P(x)=a_0+a_1x+\cdots+a_{p-1}x^{p-1}\in\mathbb{C}[x]$ such that $P(1)=a_0+a_1+\cdots+a_{p-1}=0$, and the $p$-periodic sequence $(a_n)_{n\geq 0}$, i.e., $a_{n+p}=a_n$, for all $n\in\mathbb{N}$. Then, the following series converges and \begin{equation}\label{Eq. Prop 2}
	\sum_{n=0}^\infty \frac{a_n}{n+1}=\int_0^1 \frac{P(t)}{1-t^p}dt.
	\end{equation}
\end{prop}

\begin{proof} Dirichlet's test shows convergence, as the partial sums $\sum_{k=0}^n a_k$ reduce to $a_0,a_0+a_1,\dots$,  $a_0+a_1+\cdots+a_{p-1}=0$. Then, in view of Theorem \ref{T Abel} \\ $$\sum_{n=0}^\infty \frac{a_n}{n+1}=\lim_{x\to 1^-} \sum_{n=1}^\infty a_n\frac{x^{n+1}}{n+1}=\lim_{x\to 1^-} \sum_{j=0}^{p-1} a_j \sum_{n=0}^\infty \frac{x^{np+j+1}}{np+j+1}.$$ The last equality holds because, for $0<x<1$, this is an absolutely convergent series. Now, using the uniform convergence in closed intervals of $[0,1)$, the sum above is  $$\sum_{j=0}^{p-1} a_j \sum_{n=0}^\infty \int_0^x t^{np+j}dt= \sum_{j=0}^{p-1} a_j \int_0^x \frac{t^{j}}{1-t^p}dt=\int_0^x \frac{P(t)}{1-t^p}dt.$$ Since $P(1)=0$, canceling $t-1$ from the numerator and denominator we find the integrand is continuous at $1$. Thus, we can take $x\to1^-$, leading to the result.
\end{proof}

To unravel the values in Example \ref{Ex. Knopp} (\ref{Eq. 2int I}) and \ref{Ex. Diric} we have to face the integral (\ref{Eq. Prop 2}). The calculation will be postponed to the last section, but we won't make you wait until then! Awaiting for logarithms and expressions involving $\pi$, we present:

\begin{theorem}\label{Main Thm.} If $P(x)=\sum_{l=0}^{p-1} a_l x^l\in \mathbb{C}[x]$ satisfies $P(1)=0$, then $	\int_0^1 \frac{P(t)}{1-t^p}dt=$
	$$\sum_{l=1}^{p-1} \frac{a_{l-1}}{p}\left[\ln(2p)+\frac{\pi}{2}  \cot\left(\!\frac{\pi l}{p}\right)-2\sum_{j=1}^{\lfloor\frac{p+1}{2}\rfloor-1}\!\!  \cos\left(\frac{2\pi  lj}{p}\right)\ln\left(\!\sin\left(\frac{\pi j}{p}\right)\!\right)\!\right],$$ where $\lfloor\cdot \rfloor$ denotes the floor function.
\end{theorem}

For now, let's be content with some less intimidating consequences of it. We challenge the reader to follow the hints and complete the next proof.

\begin{coro}\label{Coro. Ipq} Let $0<q<p$ be positive integers. Then, $I(p/q)$ is equal to \begin{equation}\label{Eq. I(p/q)}
	\frac{q\pi/p}{2\sin(q\pi/p)}-\frac{2q}{p}\sum_{j=0}^{\lfloor p/2\rfloor-1}\cos\left((2j+1)\frac{q\pi}{p}\right)\ln\left(\sin\left((2j+1)\frac{\pi}{2p}\right)\right).
	\end{equation} Furthermore, \begin{align}
	\label{Eq. int sin} &\int_0^1 \frac{t^{q-1}-t^{p-q-1}}{1-t^p}dt=\frac{\pi}{p} \cot(q\pi/p),\\
	\label{Eq. int sin 2} &\int_0^1 \frac{t^{q-1}+t^{p-q-1}}{1+t^p}dt=\int_0^{+\infty} \frac{t^{q-1}}{1+t^p}dt=\frac{\pi/p}{\sin(q\pi/p)}.
	\end{align} 
\end{coro}

\begin{proof} For (\ref{Eq. I(p/q)}) use Theorem \ref{Main Thm.} in (\ref{Eq. 2int I}) for  $2p$ and $t^{q-1}-t^{p+q-1}$. You may recall  that  $\cot(x)-\cot(\pi/2+x)=\cot(x)+\tan(x)=2/\sin(2x)$.  For (\ref{Eq. int sin 2}) note that $\int_0^{+\infty} \frac{t^{q-1}}{1+t^p}dt=\int_0^1  \frac{t^{q-1}}{1+t^p}dt+\int_1^{+\infty}  \frac{t^{q-1}}{1+t^p}dt$, and $\int_1^{+\infty} \frac{t^{q-1}}{1+t^p}dt=\int_0^1 \frac{s^{p-q-1}}{1+s^p}ds$ under the change of variables $t=1/s$. For the remaining apply again Theorem \ref{Main Thm.} and (\ref{Eq. I(p/q)}).
\end{proof}

Equations (\ref{Eq. int sin}) and (\ref{Eq. int sin 2}) seem innocent, but Euler used them to find the partial fraction decomposition of the reciprocal of sine and the cotangent function! \cite[p. 94-96]{Varadarajan}.

\section{The way ahead, delving into special functions}\label{Sec. Digamma}

It is time for our last worked example related to Gauss' work on the \textit{digamma function} \begin{equation}\label{Eq. Def psi}
\psi(z):=-\gamma+\sum_{k=0}^\infty \frac{1}{k+1}-\frac{1}{z+k},
\end{equation} where $\gamma$ is as in (\ref{Eq. EulerMascheroni}). Theorem \ref{T Abel} is back to prove an integral representation for $\psi$, which will show how this special function appeared hidden in Examples \ref{Ex. Knopp} and \ref{Ex. Diric}. All we need to know here is collected in:

\begin{theorem}[Gauss's Digamma Theorem] \label{Thm Gauss}  The map $\psi:\mathbb{C}\setminus\{0,-1,-2,\dots\}\to\mathbb{C}$ is well-defined and it can be written as 
	\begin{equation}\label{Eq. int psi} \psi(z)=-\gamma+\int_0^1 \frac{1-t^{z-1}}{1-t} dt,\qquad \text{ for }\, \textup{Re}(z)>0.
	\end{equation} Furthermore, for integers $0<q<p$, it assumes the value
	\begin{align}
	\label{Eq. psi pq int}\psi(q/p)+\gamma=-\ln(2p)-\frac{\pi}{2}\cot\left(\frac{\pi q}{p}\right)+2\sum_{j=1}^{\lfloor\frac{p+1}{2}\rfloor-1} \cos\left(\frac{2\pi jq}{p}\right)\ln\left(\sin\left(\frac{\pi j}{p}\right)\!\right). 
	\end{align}\end{theorem}

\begin{proof} To start with, we check that (\ref{Eq. Def psi}) converges absolutely for  any $z\neq 0, -1,-2,\dots$: fixing an integer $N>0$, if $|z|<N<k/2$, then $|z+k|\geq k-|z|>k-N>k/2$ and  $$\left|\frac{1}{k+1}-\frac{1}{z+k}\right|=\left|\frac{z-1}{(k+1)(z+k)}\right|<2(N+1)/k^2,$$ which is the general term of a convergent series independent of $z$. To prove (\ref{Eq. int psi}) fix $0<s<1$. Near $t=0$, the integrand behaves as $1-t^{z-1}$ and $\lim_{\epsilon\to 0^+}z\int_\epsilon^s t^{z-1}dt=\lim_{\epsilon\to 0^+}s^z-\epsilon^z$ exists only for $\text{Re}(z)>0$, justifying this constraint. Now, \begin{align*}
	\int_0^s \frac{1-t^{z-1}}{1-t}dt&=\int_0^s \sum_{k=0}^\infty t^k-t^{z-1+k}ds=\sum_{k=0}^\infty \frac{s^{k+1}}{k+1}-\frac{s^{z+k}}{z+k}\\&=\sum_{k=0}^\infty \left(\frac{1}{k+1}-\frac{1}{z+k}\right)s^{k+1}+(s-s^z)\sum_{k=0}^\infty\frac{s^{k}}{z+k}.
	\end{align*} The last equality is used to apply Theorem \ref{T Abel}, as it shows the first sum goes to $\psi(x)+\gamma$ as $s\to 1^-$. Now, we show the second term goes to $0$: as $|z+k|\geq \text{Re}(z+k)=\text{Re}(z)+k>k$ we can bound the second sum by  
	$$\sum_{k=0}^\infty\frac{s^{k}}{|z+k|}=\frac{1}{|z|}+\sum_{k=1}^\infty\frac{s^{k}}{|z+k|}< \frac{1}{|z|}+\sum_{k=1}^\infty\frac{s^{k}}{k}=\frac{1}{|z|}-\ln(1-s).$$ But $ (s-s^z)\ln(1-s)=s\frac{1-s^{z-1}}{1-s}(1-s)\ln(1-s)\to 0$, as $(1-s)\ln(1-s)\to 0$ and ${(1-s^{z-1})}/{(1-s)}\to\left.{d(s^{z-1})}/{ds}\right|_{s=1} =z-1$ when $s\to 1^-$, which is enough to conclude. Finally, (\ref{Eq. psi pq int}) is a direct consequence of Theorem \ref{Main Thm.} since \begin{equation*}\label{Eq. psi pq}
	\psi(q/p)+\gamma=\int_0^1 \frac{1-s^{q/p-1}}{1-s}ds=p\int_0^1 \frac{t^{p-1}-t^{q-1}}{1-t^p}dt.
	\end{equation*}
\end{proof}

\begin{remark}\label{Nota Gauss}
	Equation (\ref{Eq. psi pq int}) is entry (75) in Gauss' 1813 paper \cite{Gauss}, 13 years before Abel wrote his result. Several proofs are available in the literature. The above, including Theorem \ref{Main Thm.}, is close to \cite[p. 35-38]{Campbell} although Theorem \ref{T Abel} is not even mentioned there, in contrast with Jensen's \cite{Jensen}\cite[Thm. 1.2.7]{Andrews},\cite[p. 95, p. 498]{Knuth}. Another alternative is Lehmer's proof \cite{Lehmer} who used a generalization of Euler-Mascheroni constant for arithmetical progressions and finite Fourier series, see also Appendix B in \cite{Blagouchine}. Finally, we must mention $\psi(q/p)+\gamma$ is in fact transcendental! \cite{Murty}.
\end{remark}

With this said, we have concluded our examples, hoping to have drawn the reader's attention on how Abel's theorem finds its home in the theory of special functions. This point can be strengthened by mentioning another result in the realm of hypergeometric functions: Gauss' value for (\ref{Eq. Hyper}) at $z=1$ \cite{Gauss}, \cite[Thm. 2.2.2]{Andrews}: $${}_{2}F_{1}(a,b;c;1)=1+\sum_{n=1}^\infty \frac{(a)_n(b)_n}{(c)_n n!}={\frac {\Gamma (c)\Gamma (c-a-b)}{\Gamma (c-a)\Gamma (c-b)}},\quad \text{if Re} (c)>\text{Re} (a+b),$$ see also Kummer's formula \cite[Coro. 3.1.2]{Andrews}. Here $\Gamma$ denotes the Gamma function, usually taught before $\psi$, and both related by $\psi(z)=\Gamma'(z)/\Gamma(z)$ \cite{Jensen}.

Now, let's reap the harvest! In view of Theorems \ref{Main Thm.} and \ref{Thm Gauss} we find \vspace{-0.1cm}$$\int_0^1 \frac{P(t)}{1-t^p}dt=-\sum_{l=1}^{p-1} \frac{a_{l-1}}{p} (\gamma+\psi\left(l/p\right))=\frac{\gamma }{p}a_{p-1}-\sum_{l=1}^{p-1} \frac{a_{l-1}}{p} \psi\left(l/p\right).$$ This is particularly interesting if  $P(t)=\chi(1)+\chi(2)t+\cdots+\chi(p-1)t^{p-2}$, where $\chi$ is a non-trivial $p$-periodic Dirichlet character, as it proves the formula \begin{equation}\label{Eq. L1 psi}
L(1,\chi)=-\frac{1}{p} \sum_{j=1}^{p-1} \chi(j) \psi(j/p),
\end{equation} familiar in Number Theory \cite{Japan}. A second consequence, valid if $q$ doesn't divide $k$, is $$\sum_{j=1}^{q-1} \psi(j/q) e^{2\pi i jk/q}=\gamma+q\ln\left(2\sin(\pi k/q)\right)+i(2k-q){\pi}/{2},$$ due to Gauss. In this case,  $P(t)=\omega_q^{k}(1+\omega_q^{k} t+\omega_q^{2k}t^2+\cdots+\omega_q^{(q-1)k}t^{q-1})=\omega_q^{k} \frac{1-t^q}{1-\omega_q^k t}$, where $\omega_q=e^{2\pi i/q}$, joint with (\ref{Eq. Lema2}), do the trick.

On the other hand, the function $I(\lambda)$ that we met in Example \ref{Ex. Knopp} can easily be studied by means of $\psi$. In fact, it can be extended to $\mathbb{C}\setminus\{0,-1,-1/2,\dots\}$ as \begin{equation}\label{Eq. I int}
2\lambda I(\lambda)=\psi\left({1}/{2}+{1}/{2\lambda}\right)-\psi\left({1}/{2\lambda}\right).
\end{equation} We can prove this, by means of series or exploiting the integral representations. Using the latter, if we change variables $s=t^{2\lambda}$, $\lambda>0$ in $I(\lambda)$, we find  $2\lambda I(\lambda)$ is equal to $$\int_0^1 \frac{s^{\frac{1}{2\lambda}-1}}{1+\sqrt{s}}ds=\int_0^1 \frac{s^{\frac{1}{2\lambda}-1}(1-\sqrt{s})}{1-s}ds=\int_0^1 \frac{(1-s^{\frac{1}{2\lambda}-\frac{1}{2}})-(1-s^{\frac{1}{2\lambda}-1})}{1-s}ds,$$ as required. The singularities of $I(\lambda)$ are then located at $\lambda=0$ and $\lambda=-1/k$, for integers $k\geq 1$ coming from the singularities of $\psi$ in (\ref{Eq. I int}).

A last word on the values of $\psi$. Theorem \ref{Thm Gauss} gives them at rational numbers in the interval $(0,1)$. For the remaining, we can use the functional equation \begin{equation}\label{Eq. Feq psi}
\psi(z+1)-\psi(z)={1}/{z},
\end{equation} which follows from (\ref{Eq. Def psi}) using a telescoping series. Examples are $\psi(1)=-\gamma$ and $\psi(m)=-\gamma+1+\frac{1}{2}+\cdots+\frac{1}{m-1}$, for integers $m\geq 2$. In the same way, Corollary \ref{Coro. Ipq} gives the vales of $I$ for rational numbers greater than $1$, and (\ref{Eq. I int}) recovers them all. But we can also use the functional equation $
(1-\lambda)I(\lambda)+I\left(\frac{\lambda}{1-\lambda}\right)=1$. We trust the reader to check it, using the integral representation or equations (\ref{Eq. I int}) and (\ref{Eq. Feq psi}).

\section{Theorem \ref{Main Thm.} and some calculations}

The idea of this proof goes back to Euler! \cite[p. 96]{Varadarajan}. It is simple, up to some trigonometric identities spinning around the different proofs of Gauss' digamma theorem. 

\begin{proof}[Sketch of proof] This is a guided exercise for the reader to fill in the gaps! and the goal is to show the definite integral is equal to the equivalent form \begin{equation}\label{Eq. Main aux}
	-a_{p-1}\frac{\ln(2p)}{p}+\sum_{l=1}^{p-1} \frac{a_{l-1}}{p}\left[\frac{\pi}{2} \cot\left(\!\frac{\pi l}{p}\right)-2\sum_{j=1}^{\lfloor\frac{p+1}{2}\rfloor-1}\!\!  \cos\left(\frac{2\pi lj}{p}\right)\ln\left(\!\sin\left(\frac{\pi j}{p}\right)\!\right)\!\right].
	\end{equation} We fix $\omega_p:=e^{2\pi i/p}$. To start, it is enough to assume $P$ has real coefficients (why?). We first take care of the case $p=2m$ from which the case for $p=2m+1$ will follow.	
	
	\
	
	\textit{Step 1: Partial fractions.} In general, if $R$ and $Q$ are polynomials, $\text{deg}(R)<\deg(Q)=r$, and $Q$ has distinct roots $\lambda_1,\dots,\lambda_r$, then  \begin{equation}\label{Lema PQ} 
	\frac{R(t)}{Q(t)}=\sum_{j=1}^r \frac{R(\lambda_j)}{Q'(\lambda_j)}\frac{1}{t-\lambda_j}.
	\end{equation} We can use this with $R(t)=P(t)/(1-t)$ and $Q(t)=1+t+\cdots+t^{p-1}$  (why?). Now, check $Q$ has roots $-1,\omega_p^{\pm 1},\dots, \omega_p^{\pm(m-1)}$, and $\frac{P(\omega_p^j)}{(1-\omega_p^j)Q'(\omega_p^j)}=-\omega_p^{j}P(\omega_p^{j})/p.$

	\textit{Step 2: Integrating.}  First note,   $\overline{e^{i\theta}P(e^{i\theta})\int_0^1 \frac{dt}{t-e^{i\theta}}}=e^{-i\theta}P(e^{-i\theta})\int_0^1 \frac{dt}{t-e^{-i\theta}}$, where the bar denotes complex conjugate. Use this, joint with (\ref{Lema PQ}) and (\ref{Eq. Lema2}) to check that \begin{align}
	\label{Eq. Proof}\int_0^1&\frac{P(t)}{1-t^p}dt=	\frac{\ln(2)}{p}\left[P(-1)-2\sum_{j=1}^{m-1} \text{Re}(\omega_p^{j}P(\omega_p^{j}))\right]\\
	\nonumber &-\frac{2}{p}\sum_{j=1}^{m-1} \text{Re}(\omega_p^{j}P(\omega_p^{j}))\ln\left(\sin\left(\frac{\pi j}{p}\right)\!\right)+\frac{\pi}{p^2}\sum_{j=1}^{m-1} \text{Im}(\omega_p^{j}P(\omega_p^{j}))(p-2j).
	\end{align}
	
	\textit{Step 3: Simplifying I.} Writing $e^{i\theta} P(e^{i\theta})=\sum_{l=1}^{p} a_{l-1} (\cos(l\theta)+i\sin(l\theta))$, check that the term multiplying $\ln(2)/p$ above is equal to\begin{align*}
	\sum_{l=1}^{p} a_{l-1} \left[(-1)^{l-1}-\sum_{j=1}^{m-1} 2\cos(2\pi jl/p)\right]=-pa_{p-1}.
	\end{align*} For this, consider when $l$ is even or odd, and use the real part of equation (\ref{Eq. sum eit}), namely, \begin{equation}\label{Eq. sum cos} \sum_{j=1}^k \cos(j\theta)=\frac{\sin\left(k\theta/2\right)}{\sin\left(\theta/2\right)}\cos\left((k+1)\frac{\theta}{2}\right).
	\end{equation}  Finally, to deal with the last sum in (\ref{Eq. Proof}), if $\theta_0={2\pi l}/{p}={\pi l}/{m}$, then \begin{align*}
	2\sum_{j=1}^{m-1} (m-j)\sin\left(\theta_0 j\right)=2\sum_{j=1}^{m-1} j\sin\left(\theta_0(m-j)\right)=2(-1)^{l+1}\sum_{j=1}^{m-1} j\sin\left(\theta_0j\right),\end{align*} since $\sin(m\theta_0)=0$ and $\cos(m\theta_0)=(-1)^{l}$. Prove this sum is equal to $m\cot\left(\theta_0/2\right)$ if $l<p$, and $0$ if $l=p$. For instance, take the derivative w.r.t. $\theta$ in (\ref{Eq. sum eit}), consider the imaginary part of the result, and run the computations for $l$ even or odd.
	
	\textit{Step 4: Simplifying II.} Collecting the calculations so far, (\ref{Eq. Proof}) should have reduced to (\ref{Eq. Main aux}), except for the term multiplying $-a_{p-1}$ which is $
	\ln(2)-\frac{2}{p}\sum_{j=1}^{m-1} \ln\left(\sin\left(\frac{\pi j}{p}\right)\!\right)$. To achieve the final result use the product 
	\begin{equation}\label{Eq. psin}\prod_{j=1}^{m-1} \sin\left(\frac{\pi j}{2m}\right)=\frac{\sqrt{m}}{2^{m-1}}.\end{equation} This formula is a consequence of the better known\footnote{ $T(x)=x^N-1=(x-1)(x-\omega_N)\cdots (x-\omega_N^{N-1})$ satisfies $(1-\omega_N)\cdots (1-\omega_N^{N-1})=T'(1)=N$.
		The formula then follows by taking the modulus on both sides, as $|1-\omega_N^j|=2\sin(\pi j/N)$, see (\ref{Eq. mod 1-ei0}).} $
	\prod_{j=1}^{N-1} \sin\left(\frac{\pi j}{N}\right)=\frac{N}{2^{N-1}}\,$. Indeed, if $N=2m$, write this product as $\prod_{j=1}^{m-1} \sin\left(\frac{\pi j}{2m}\right)\cdot \prod_{j=1}^{m-1} \sin\left(\frac{\pi (m+j)}{2m}\right)=\frac{m}{4^{m-1}}.$ Taking into account that  $\sin\left(\pi/2+\theta\right)=\cos(\theta)=\sin\left(\pi/2-\theta\right)$, $$\prod_{j=1}^{m-1} \sin\left(\frac{\pi (m+j)}{2m}\right)=\prod_{j=1}^{m-1} \sin\left(\frac{\pi (m-j)}{2m}\right)=\prod_{j=1}^{m-1} \sin\left(\frac{\pi j}{2m}\right),$$ which is enough to conclude (\ref{Eq. psin}).
	
	\
	
	\textit{Step 5: $p=2m+1$.}	Make the change of variables $t=s^2$, and apply the even case to find that $\int_0^1 \frac{P(t)}{1-t^{p}}dt=2\int_0^1 \frac{sP(s^2)}{1-s^{2p}}ds$ is equal to \begin{align*}\sum_{l=1}^{p-1} \frac{a_{l-1}}{p} \left[\ln(4p)+\frac{\pi}{2} \cot\left(\!\frac{\pi l}{p}\!\right)-2\sum_{j=1}^{p-1} \cos\left(\frac{2\pi lj}{p}\right) \ln\left(\!\sin\left(\frac{\pi j}{2p}\right)\!\right)\right].
	\end{align*} To correct the inner sum, put the terms indexed by $j$ and $2m-(j-1)=p-j$ together, $j=1,\dots,m$. Then, use $\sin(\pi/2-\theta)=\cos(\theta)$, $\sin(2\theta)=2\sin(\theta)\cos(\theta)$, and the sum $2\sum_{j=1}^{m} \cos\left(\frac{2\pi jl}{2m+1}\right)=-1$ justified using (\ref{Eq. sum cos}), to show the inner sum is ${\ln(2)}/{2}+\sum_{j=1}^{m} \cos\left({2\pi lj}/{p}\right) \ln\left(\sin\left({\pi j}/{p}\right)\right)$, what leads to the final result.

\end{proof}

Theorem \ref{Main Thm.} and Corollary \ref{Coro. Ipq} can easily be applied to small values of $p$. For instance, \begin{align*}1-\frac{1}{4}+\frac{1}{7}-\frac{1}{10}+\cdots=
\int_0^1\frac{dt}{1+t^3}=\frac{\pi/3}{2\sin(\pi/3)}-\frac{2}{3}\cos\left(\pi/3\right)\ln(\sin(\pi/6)),
\end{align*} which gives ${\ln (2)}/{3}+ {\sqrt{3}\pi}/{9}$. However, it is a different story for larger values of $p$. Consider for instance $p=5$, still manageable by hand, for which 
$$e^{2\pi i/5}=\frac{\sqrt{5}-1}{4}+i\frac{\sqrt{5+\sqrt{5}}}{2\sqrt{2}},\quad e^{i\pi/5}=\frac{\sqrt{5}+1}{4}+i\frac{\sqrt{5-\sqrt{5}}}{2\sqrt{2}}.$$ The trick is observing $e^{2\pi i/5}$ satisfies $(z+1/z)^2+(z+1/z)=1$, which is solved using the quadratic formula twice. Paying attention to signs we get the correct root, and then $e^{i\pi/5}$ is extracted from the double-angle formulas. After this, the values $\psi(j/5)$, $j=1,2,3,4$ can be found \cite[p. 147]{Jensen} and results such as \begin{align*}
1-\frac{1}{2}-\frac{1}{3}+\frac{1}{4}+\frac{1}{6}-\frac{1}{7}-\cdots=\int_0^1 \frac{1-t-t^2+t^3}{1-t^5}dt=\frac{2}{\sqrt{5}}\ln\left(\frac{1+\sqrt{5}}{2}\right),
\end{align*} are achieved. In general, these calculations are possible for $p=2^rp_1\cdots p_t$, where $p_l=2^{(2^{m_l})}+1$ are distinct Fermat primes, because $\omega_p^j$ can be found explicitly using square roots. In fact, this is the condition for the regular polygon with $p$ sides to be constructible with ruler and compass! Despite this, things can get out of hands... just recall Gauss' formula for $\cos(2\pi/17)$ in the construction of the heptadecagon \cite{ConwayGuy}.

Computations as in Theorem \ref{Main Thm.}
were used by Dirichlet to find $L(1,\left(\frac{\cdot}{\cdot}\right))$, for the Legendre symbol \cite{Apostol}, in his famous \textit{Vorlesungen \"{u}ber Zahlentheorie} published posthumously by Dedekind in 1863 \cite[$\S$103]{Dirichlet}. For instance, he found \begin{equation}\label{Eq. pi/7}
\sum_{n=1}^\infty \left(\frac{n}{p}\right)\frac{1}{n}=-\frac{i^{(p-1)^2/4}}{\sqrt{p}}\sum_{j=1}^{p-1} \left(\frac{j}{p}\right)\left(\ln\left(\sin\left(\frac{j\pi}{p}\right)\right)-\frac{\pi j}{p}i\right),
\end{equation} valid for any odd, square free $p>1$. If $p=7$, we have  $\left(\frac{1}{7}\right)=\left(\frac{2}{7}\right)=\left(\frac{4}{7}\right)=1$, and $\left(\frac{3}{7}\right)=\left(\frac{5}{7}\right)=\left(\frac{6}{7}\right)=-1$. Therefore, since (\ref{Eq. pi/7}) is real,$$1+\frac{1}{2}-\frac{1}{3}+\frac{1}{4}-\frac{1}{5}-\frac{1}{6}+\cdots=\frac{\pi}{\sqrt{7}},$$ far from what Theorem \ref{Main Thm.} gives... To reconcile both answers we would  need relations such as $\cot\left({\pi}/{7}\right)+\cot\left({2\pi}/{7}\right)-\cot\left({3\pi}/{7}\right)=\sqrt{7}$, see \cite{Glaisher} for further discussion, and \cite[p. 278--280]{Stein} for the use of finite Fourier series. The key here is Gauss' sums unveiling these relations, but this is already part of another story.



\begin{thebibliography}{99}

\bibitem{Abel}{N. H. Abel}, \emph{Untersuchungen \"{u}ber die Reihe $1+\frac{m}{1}x+\frac{m\cdot(m-1)}{1\cdot 2}x^2+\frac{m\cdot(m-1)(m-2)}{1\cdot 2\cdot 3}x^3+\dots$ u.s.w.}, J. Reine Angew. Math., {\bf 1(4)} (1826), 311--339.

\bibitem{Ahlfors}{L. Ahlfors,} 
\emph{Complex Analysis, an Introduction to the Theory of Analytic Functions of One Complex Variable}, Third edition, International series in pure and applied mathematics, McGraw Hill, New York, 1979.

\bibitem{Andrews}{G. E. Andrews, R. Askey and R. Roy,} \emph{Special Functions}, Encyclopedia of Mathematics and its Applications, Cambridge University Press, Cambridge UK, 1999. 

\bibitem{Apostol}{T. M. Apostol,}  \emph{Introduction 	to Analytic Number Theory}, Springer-Verlag, Berlin Heidelberg, 1976.

\bibitem{Ayoub}{R. Ayoub,} \emph{Euler and the Zeta Function}, Amer. Math. Monthly, {\bf 81(10)} (1974), 1067--1086. 

\bibitem{Blagouchine}{I. Blagouchine,} \emph{A theorem for the closed-form evaluation of the first generalized Stieltjes constant at rational arguments and some related summations}, J. Number Theory, {\bf 148} (2015), 537--592. 

\bibitem{Bromwich}{T. J. Bromwich,} \emph{An Introduction to the Theory of Infinite Series}, Third edition, AMS/Chelsea Publishing Series, AMS. Reprinted from the 1908 first edition, Macmillan and Co., Limited, London. 2005.

\bibitem{Campbell}{R. Campbell,} \emph{Les int\'{e}grales eul\'{e}riennes et leurs applications}, Dunod, Paris, 1966.

\bibitem{ConwayGuy}{J. H. Conway and R. Guy,} \emph{The Book of Numbers}, Copernicus Series, Springer Science \& Business Media, New York, 1998.

\bibitem{Alternating}{C. C. Cowen, K. R. Davidson and R. P. Kaufman,} \emph{Rearranging the alternating harmonic series}, Amer. Math. Monthly, {\bf 87(10)} (1980), 817--819. 

\bibitem{Coolidge}{J. L. Coolidge,} \emph{The story of the binomial theorem}, Amer. Math. Monthly, {\bf 56(3)} (1949), 147--157. 

\bibitem{Dirichlet}{P. G. L. Dirichlet and R. Dedekind,} \emph{Lectures on Number Theory (Vorlesungen \"{u}ber Zahlentheorie)}, In: History of mathematics, {\bf 16} 1999. Translated by J. Stillwell. AMS/London Math. Society, 1894.

\bibitem{Gauss}{C. F. Gauss,} \emph{Disquisitiones generales circa seriem infinitam $1+\frac{\alpha\beta}{1\cdot \gamma}x+\frac{\alpha(\alpha+1)\beta(\beta+1)}{1\cdot 2\cdot  \gamma(\gamma+1)}xx+etc$}, {Comm. soc. regine scient. Gott.}, {\bf 2} (1866). Reprinted Werke III (1813), 123--162.

\bibitem{Glaisher}{J. W. L. Glaisher,} \emph{On series for ${n\pi}/{\sqrt{P}}$}, Messenger of Maths., {\bf 31} (1901), 98--115.

\bibitem{Hardy1907}{G. H. Hardy,} \emph{Some theorems connected with Abel's theorem on the continuity of power series}, P. Lond. Math. Soc., {\bf s2-4 (1)} (1907), 247--265. 

\bibitem{Japan}{M. Hashimoto, S. Kanemitsu and M. Toda,}  \emph{On Gauss' formula for $\psi$ and finite expressions for the $L$-series at $1$}, J. Math. Soc. Japan, {\bf 60(1)} (2008),  219--236. 


\bibitem{Jensen}{J. L. W. V. Jensen,} \emph{An elementary exposition of the theory of the Gamma function}, {Ann. Math.}, {\bf 17} (1915), 124--166. 

\bibitem{Knopp}{K. Knopp,}
\emph{Theory and Application of Infinite Series},  Dover Publications, Inc., New York, 1990.

\bibitem{Knuth}{D. E. Knuth,}  \emph{The art of computer programming}. Vol. 1: Fundamental algorithms, Third edition, Addison-Wesley, Reading, MA, USA, 1997.

\bibitem{Lehmer}{D. H. Lehmer,} 
\emph{Euler constants for arithmetical progressions},
{Acta Arith.}, {\bf 27} (1975), 125--142. 

\bibitem{Mark}{A. I. Markushevich,} \emph{Theory of Functions of a Complex Variable}, Vol. 1. Revised English edition translated and edited by R. A. Silverman, Prentice Hall, Inc., Englewood Cliffs, NJ, 1965.

\bibitem{Murty}{M. Ram Murty and N. Saradha,}
\emph{Transcendental values of the digamma function}, {J. Number Theory,} {\bf 125(2)} (2007), 298--318. 

\bibitem{Rudin real}{W. Rudin,} 
\emph{Principles of mathematical analysis}, Third edition.  International Series in Pure \& Applied Mathematics, McGraw Hill, 1976.

\bibitem{Sorensen}{H. K. S\o rensen,} \emph{Exceptions and counterexamples: Understanding Abel’s comment on Cauchy’s Theorem}, {Hist. Math.}, {\bf 32} (2005), 453--480. 

\bibitem{Stedall}{J. Stedall,}  \emph{Mathematics emerging. A sourcebook 1540--1900}, Oxford University Press. 2008.

\bibitem{Stein}{E. M. Stein and R. Shakarchi,} \emph{Fourier Analysis: An Introduction}. Princeton Lectures in Analysis {\bf I}, Princeton University Press, Princeton, NJ, 2003.

\bibitem{Varadarajan}{V. S. Varadarajan,} 
\emph{Euler through time: a new look at old themes}, AMS, 2006.

 \end{thebibliography}
\end{document}